\documentclass[11pt,CJK,a4paper]{article}
\input amssymb.sty
\usepackage{amsmath}
\usepackage{amssymb}
\usepackage{bbm}
\usepackage{comment}  
\usepackage{extarrows}

\numberwithin{equation}{section}  

\evensidemargin=\oddsidemargin\textwidth=142mm \textwidth=160mm
\newtheorem{theorem}{Theorem}[section]
\newtheorem{lemma}{Lemma}[section]

\newenvironment{proof}[1][Proof]{\begin{trivlist}
\item[\hskip \labelsep {\bfseries #1}]}{\end{trivlist}}

\def \cp{C_{\varphi}}
\def \p{\varphi}
\def\z{\mathcal{Z}}

\def\uc{uC_\varphi}

\def\fz{\|f\|_{\z}}

\def\a{\alpha}
\def\b{\mathcal{B}}
\def\ba{\b^{\a}}

\def\bao{\ba_0}

\def \cp{C_{\varphi}}
\def \p{\varphi}

\def \H{H ^\infty}

\def\uc{uC_\varphi}

\def\ci{C_\varphi V_g}
\def\z{\mathcal{Z}}

\def\fz{\|f\|_{{\z}^\alpha}}

\def\ic{V_g C_\varphi }
\def\cj{C_\varphi U_g}
\def\jc{U_g C_\varphi }
\def\B{\mathcal{B}}

\def \hv{H_{v_1}^\infty}
\def \hw{H_w^\infty}
\def\za{\mathcal{Z}^{\alpha}}
\def\zb{\mathcal{Z}^{\beta}}
\def\zao{\mathcal{Z}^{\alpha}_0}
\def \hvb{H_{v_\beta}^\infty}
\def \hva{H_{v_\alpha}^\infty}
\def \va{v_\alpha}
\def \vb{v_\beta}
\def\zz{\mathcal{Z}^2}
\def\da{D_\alpha}
\def\sa{S_\alpha}

\def\db{D_\beta}
\def\sb{S_\beta}

\def\ba{\b^{\a}}

\begin{document}
\title {  Essential norms of Volterra-type operators between~$Zygmund$~ type spaces\footnote{The research was supported by the National Natural Science Foundation of China (Grant No. 11571217) and
the Natural Science Foundation of Fujian Province, China(Grant No. 2015J01005).}}
\author{  Shanli Ye\footnote{Corresponding author.~ E-mail address:  ye\_shanli@aliyun.com; slye@zust.edu.cn}\\
(\small \it School of Sciences, Zhejiang University of Science and Technology,
Hangzhou 310023, China) \\
Caishu Lin\\
(\it \small Department of Mathematics, Fujian Normal University, Fuzhou 350007, China)}

 \date{}
\maketitle
\begin{abstract} ~In this paper, we investigate the boundedness of some Volterra-type operators between ~$Zygmund$~ type spaces. Then, we give
 the essential norms of such operators in  terms of ~$g,\varphi$, their derivatives and the n-th power ~$\varphi^n$ of ~$\varphi$. 
{\small\bf 2010 MR Subject Classification }  47B33, 47B38, 30H99.\\

\end{abstract}
\maketitle

\section{Introduction}
  Let~ $D=\{z:|z|<1\}$ be the open unit disk in the complex plane
  $\mathbf{C}$ and $\partial{D}=\{z:|z|=1\}$ be its boundary, and $H(D)$ denote the set of all analytic functions on
  $D$.

  For every $0<\alpha<\infty$, we denote by  $\ba$ the Bloch type space of all
  functions $f\in H(D)$ satisfying
$$
     b^{\a}(f)=\sup_{z\in D} (1-|z|^2)^\alpha|f'(z)|<\infty
$$
 endowed with the norm $\|f\|_{\ba}=|f(0)|+ b^{\a}(f) $.  The little Bloch type
 space $\bao$ consists of all $f\in \ba$ satisfying $
\lim_{|z|\to 1^-}(1-|z|^2)^\alpha|f'(z)|= 0$, ~$\bao$ is obviously the closed subspace of $\za$. When  $\alpha=1$, we get the classical Bloch space~$\b^1=\b$ and little Bloch space~$\b^1_0=\b_0$.
It is well-known that for $0<\alpha<1$, $\ba$ is a subspace of $\H$, the Banach space
of bounded analytic functions on $D$.

   For~$0<\alpha<\infty$ we denote by  $\za$ the
$Zygmund~~  type~~ space$ of those functions $f\in H(D)$ satisfying
$$
     \sup_{z\in D} (1-|z|^2)^\alpha|f''(z)|<\infty,
$$
 and the little Zymund type
 space $\zao$ consists of all $f\in \za$ satisfying $
\lim_{|z|\to 1^-}(1-|z|^2)^\alpha|f''(z)|= 0$, ~$\zao$ is obviously the closed subspace of $\za$.
It can easily proved
that $\za$ is a Banach space under the norm
\begin{equation}
\fz=|f(0)|+|f'(0)|+ \sup_{z\in D} (1-|z|^2)^\alpha|f''(z)|
\end{equation}
 and that $\zao$ is a closed
subspace of $\za$. When ~$\alpha=1$, we get the classical Zygmund space ~$\z^1=\z$ and the little Zygmund space ~$\z^1_0=\z_0$. It is clear that $f\in \z$  if  and only if $f'\in \b^1$.

  We consider the weighted Banach spaces of analytic functions~$$H_v^\infty=\{f\in H(D):\|F\|_v=\sup_{z\in D}v(z)|f(z)|<\infty\}$$ endowed with norm~$\|\cdot\|_v$, where the weight ~$v:D\rightarrow R_+$~ is a continuous, strictly positive and bounded function. The weight ~$v$~ is called radial, if ~$v(z)=v(|z|)~$ for all ~$z\in D.$ For a weight ~$v$~ the associated weight~$\widetilde{v}$ is definded by $$
\widetilde{v}(z)=(\sup\{|f(z)|:f\in H_v^\infty ,\|f\|_v\leq1\}^{-1},~~~~~z\in D.$$
We notice the standard weights ~$v_\alpha(z)=(1-|z|^2)^\alpha$~, where~$0<\alpha<\infty$~ and it is well-known that ~$\widetilde{v_\alpha}=v_\alpha$. We also consider the logarithmic weight ~$$v_{\log}=(\log\frac{2}{1-|z|^2})^{-1},~~~~z\in D.$$It is straightforward to show that~$\widetilde{v_{\log}}=v_{\log}$.

Let $u$ be a fixed analytic function on $D$ and an analytic self-map $\varphi:D\rightarrow D$. Define
a linear operator $uC_\varphi$ on the space of analytic functions
on $D$, called a weighted composition operator, by $uC_\varphi f
=u \cdot (f\circ \varphi),$ where $f$ is an analytic function on
$D$. It is interesting to provide a function theoretic characterization when $\varphi$ and $u$ induces a bounded or compact composition operator on various function spaces. Some results on the boundedness and compactness of concrete operators between some spaces of analytic functions one of which is of Zygmund-type can be found, e.g., in $(\cite{km,li1,li2,li3,ste1,ste2,ste3,ste4,ysl1,ysl4,ysl5})$.

Suppose that $g: D\rightarrow \mathbb{C}$ is a analytic map. Let
$U_g$ and $V_g$ denote the Volterra-type operators  with
 the analytic symbol $g$ on $D$ respectively:

 $$U_g f(z)=\displaystyle\int_0^z f(\xi)g'(\xi) d\xi,~~~~~~~~~~~~~~z\in D.$$
and
 $$V_g f(z)=\displaystyle\int_0^z f'(\xi)g(\xi) d\xi,~~~~~~~~~~~~~~z\in D.$$

If ~$g(z)=z$, then~$U_g$~is an integral operator. While ~$g(z)=\ln\frac{1}{1-z}$, then ~$U_g$ is Ces\`{a}ro operator.
Pommerenke introduced the ~$Volterra$~type operator~$U_g$~ and characterized the boundedness of~$U_g$ between~$H^2$ spaces in ~\cite{cs}~.
More recently, boundedness and compactness of~$Volterra$~type operators between several spaces of analytic functions have been studied by many authors, one may see in ~ $\cite{wol, ysl3}$.

  In this paper, we consider the following integral-type operators, which were introduced by Li and Stevi$\grave{c}$ (see e. g. \cite{li2,li4}),  they can be defined by:
$$(\cj f)(z)=\displaystyle\int_0^{\varphi(z)} f(\xi)g'(\xi) d\xi,~~~~~~~~~~~(\ci f)(z)=\displaystyle\int_0^{\varphi(z)} f'(\xi)g(\xi) d\xi,$$
$$(\jc f)(z)=\displaystyle\int_0^z f(\varphi(\xi))g'(\xi) d\xi,~~~~~~~~~~~~~(\ic f)(z)=\displaystyle\int_0^z f'(\varphi(\xi))g(\xi) d\xi.$$
We will characterize  the boundedness of those integral-type operators between Zygmund type spaces, and also estimate their essential norms. The boundedness and compactness of these operators on the logarithmic Bloch space have been characterized in \cite{ysl3}.

Recall that essential norm ~$\|T\|_{e, X\to Y}$ of a bounded linear operator ~$T: X\to Y$~ is defined as the distance from $T$ to $  \mathbf{K}(X,Y)$, the space of compact operators from $X$ to $Y$, namely,
~$$\|T\|_{e, X\to Y}=\inf\{\|T+K\|_{X\rightarrow Y}:K:X\rightarrow Y is~~ compact\}.$$
It provides a measure of non-compactness of T. Clearly, T is compact if and only if ~$\|T\|_{e, X\to Y}=0$.

Throughout this paper, constants are denoted by $C$, they are
positive and may differ from one occurrence to the other. The notation ~$a\asymp b$  means that there are  positive constants ~$C_1, C_2$  such that~$C_1 a\leq b\leq C_2a$.

\section{Boundedness}

  In order to prove the main results of this paper. we need some
  auxiliary results.

  \begin{lemma}
~$(\mbox{\rm\cite{km, ste2}})$~
For ~$0<\alpha<2$~ and let ~$\{f_n\}$~ be a bounded sequence in ~$\za$~ which converges to $0$ uniformly on compact subsets of ~$D$~. Then ~$\lim_{n\rightarrow\infty}\sup_{z\in D}|f_n(z)|=0$.
\end{lemma}

\begin{lemma}
~$(\mbox{\rm\cite{km, ste2}})$~
For every ~$f \in \za$~ where ~$\alpha>0$~ we have:

(i)~~$|f'(z)|\leq\frac{2}{1-\alpha}\|f\|_{\za}$ and $|f(z)|\leq\frac{2}{1-\alpha}\|f\|_{\za}$ for every ~$0<\alpha<1$;

(ii)~~$|f'(z)|\leq 2\log\frac{2}{1-|z|}\|f\|_{\za}$ and $|f(z)|\leq \|f\|_{\za}$ for ~$\alpha=1$;

(iii)~~$|f'(z)|\leq\frac{2}{1-\alpha}\frac{\|f\|_{\za}}{(1-|z|)^{\alpha-1}}$, for every ~$\alpha>1$;

(iv)~~$|f(z)|\leq\frac{2}{(\alpha-1)(2-\alpha)}\|f\|_{\za}$, for every $1<\alpha<2$;

(v)~~$|f(z)|\leq2 \log\frac{2}{1-|z|}\|f\|_{\za}$, for every $\alpha=2$;

(vi)~~$|f(z)|\leq\frac{2}{(\alpha-1)(\alpha-2)}\frac{\|f\|_{\za}}{(1-|z|)^{\alpha-2}}$, for every $\alpha>2$.
\end{lemma}

\begin{lemma}
~$(\mbox{\rm\cite{km}})$~
Let~$0<\alpha<\infty$, $v$~ a radial, non-increasing weights tending to $0$ at boundary of ~$D$~ and the weighted composition operator ~$\uc:\mathcal{Z}_\alpha\rightarrow H_v ^\infty$~ be bounded.

(i)~~If $0<\alpha<2$, then~$\uc$~is a compact operator.

(ii)\begin{eqnarray*}
\|\uc\|_{e,\mathcal{Z}_2\rightarrow H_v ^\infty}&\asymp&\displaystyle\limsup_{n\rightarrow\infty}(\log n)\|u{\p}^n\|_v\\
&\asymp&\displaystyle\limsup_{|\p(z)|\rightarrow1}v(z)|u(z)|\log\frac{2}{1-|\p(z)|^2}.\\
\end{eqnarray*}

(iii)~~If $\alpha>2$, then
\begin{eqnarray*}
\|\uc\|_{e,\mathcal{Z}_\alpha\rightarrow H_v ^\infty}&\asymp&\displaystyle\limsup_{n\rightarrow\infty}(n+1)^{\alpha-2}\|u{\p}^n\|_v\\
&\asymp&\displaystyle\limsup_{|\p(z)|\rightarrow1}\frac{v(z)|u(z)|}{(1-|\p(z)|^2)^{\alpha-2}}.\\
\end{eqnarray*}

\end{lemma}

The following lemma is due to ~$\mbox{\rm\cite{ho}}$~ and  ~$\mbox{\rm\cite{m}}$.

\begin{lemma}

Let ~$v$ and ~$w$~ be radial, non-increasing weights tending to zero at the boundary of $D$. Then

(i) The weighted composition operator ~$\uc$~ maps~$H_v^\infty$ into ~$\hw$ if and only if $$sup_{n\geq 0}\frac{\|u\varphi^n\|_w}{\|z^n\|_v}\asymp \sup_{z\in D}\frac{w(z)}{\widetilde{v}(\p(z))}|u(z)|<\infty,$$
with norm comparable to the above supermum.

(ii)$$\|u\cp\|_{e,H_v^\infty\rightarrow \hw}=\limsup_{n\rightarrow\infty}\frac{\|u\varphi^n\|_w}{\|z^n\|_v}=\limsup_{
|\p(z)|\rightarrow1}\frac{w(z)}{{\widetilde{v}(\p(z))}}|u(z)|.$$
\end{lemma}

\begin{lemma}
~$(\mbox{\rm\cite{ho1}})$~
For every $0<\alpha <\infty$, we have

\begin{equation}
\lim_{n\rightarrow \infty}(n+1)^\alpha\|z^n\|_{v_\alpha}=(\frac{2\alpha}{e})^\alpha
\end{equation}
and
\begin{equation}
\lim_{n\rightarrow \infty}(\log n)\|z^n\|_{v_{\log}}=1.
\end{equation}
\end{lemma}

  \begin{theorem}
   Let $\varphi$ be an analytic self-map of $D$  and~$g\in H(D)$.

   (i)~~If ~$0<\alpha<1$, then $\ic:\za \rightarrow \zb $~is a bounded operator if and only if ~$g'\in \hvb$ and
   \begin{equation}
   \sup _{n\geq 0}(n+1)^\alpha \|g(\p'){\p}^n\|_{\vb}\asymp \sup _{z \in D}\frac{(1-|z|^2)^\beta}{(1-|\p(z)|^2)^\alpha}|g(z)\p'(z)|<\infty.
   \end{equation}

   (ii)~~If ~$\alpha=1$, then  $\ic:\za \rightarrow \zb $~is a bounded operator if and only if
   \begin{equation}
   \sup _{n\geq 0}(n+1) \|g(\p'){\p}^n\|_{\vb}\asymp \sup _{z \in D}\frac{(1-|z|^2)^\beta}{1-|\p(z)|^2}|g(z)\p'(z)|<\infty
   \end{equation}
   and
   \begin{equation}
   \sup _{n\geq 0}(\log n) \|g'{\p}^n\|_{\vb}\asymp \sup _{z \in D}(1-|z|^2)^\beta \log (\frac{2}{1-|\p(z)|^2})|g'(z)|<\infty.
   \end{equation}

(iii)~~If ~$\alpha>1$, then $\ic:\za \rightarrow \zb $~is a bounded operator if and only if
 \begin{equation}
 \sup _{n\geq 0}(n+1)^\alpha \|g(\p'){\p}^n\|_{\vb}\asymp \sup _{z \in D}\frac{(1-|z|^2)^\beta}{(1-|\p(z)|^2)^\alpha}|g(z)\p'(z)|<\infty
 \end{equation}
 and
  \begin{equation}
    \sup _{n\geq 0}(n+1)^{\alpha-1} \|g'{\p}^n\|_{\vb}\asymp \sup _{z \in D}\frac{(1-|z|^2)^\beta}{(1-|\p(z)|^2)^{\alpha-1}}|g'(z)|<\infty.
    \end{equation}
  \end{theorem}

  \begin{proof}
  Suppose that ~$\ic$ is bounded from $\za$ to $\zb$. Using the test functions $f(z)=z$ and $f(z)=z^2$, we have
  $$
  (1-|z|^2)^\beta|(\ic z)''|=(1-|z|^2)^\beta|g'(z)|<\infty,
 $$
  and $$(1-|z|^2)^\beta|(\ic z^2)''|=(1-|z|^2)^\beta|2\p'(z)g(z)+2\p(z)g'(z)|<\infty.$$
    Since ~$\p$ is a self-map, we get that ~$g' \in \hvb$, ~$\p'g \in \hvb$.

  For every ~$0<\alpha<\infty $ and given nonzero ~$a \in D$. We take the test functions
     \begin{equation}
    f_a(z)=\frac{1}{\overline{a}}[\frac{(1-|a|^2)^2}{(1-\overline{a}z)^\alpha\alpha}-\frac{1-|a|^2}{(1-\overline{a}z)^{\alpha-1}}],
    \end{equation}

   \begin{equation}
    h_a(z)=\frac{1}{\overline{a}}\int_0 ^z \frac{1-|a|^2}{(1-\overline{a}w)^\alpha}dw,
    \end{equation}
  \begin{equation}
    g_a(z)= f_a(z)- h_a(z),
    \end{equation}
  for every ~$z \in D$. One can show that $f_a, h_a$, and $g_a$ are in $\zao$, $\sup_{\frac12<a<1}\|f_a\|_{\za}<\infty$ and $\sup_{\frac12<a<1}\|h_a\|_{\za}<\infty$.
     Since $g_a'(a)=0,~ g_a''(a)=\frac{\alpha}{(1-|a|^2)^{\alpha}}$, it follows that for all $z\in D$ with $|\p(z)|>\frac12$, we have
   \begin{eqnarray*}
+\infty>C\|g_a\|_{\za}&\geq &\|\ic(g_{\p(z)})\|_{\zb}\\
 &\geq& (1-|z|^2)^\beta|\p'(z)g(z)||g''_{\p(z)}(\p(z))|-(1-|z|^2)^\beta|g'(z)||g'_{\p(z)}(\p(z))|\\
&=&(1-|z|^2)^\beta|\p'(z)g(z)||\frac{\alpha}{(1-|\p(z)|^2)^\alpha}|.\\
\end{eqnarray*}
Then
\begin{eqnarray*}
\displaystyle\sup_{z \in D}|\p'(z)g(z)|\frac{(1-|z|^2)^\beta}{(1-|\p(z)|^2)^\alpha}&\leq& \sup_{|\p(z)|\leq\frac 12}\frac{|\p'(z)g(z)|(1-|z|^2)^\beta}{(1-|\p(z)|^2)^\alpha}+\sup_{|\p(z)|>\frac 12}\frac{|\p'(z)g(z)|(1-|z|^2)^\beta}{(1-|\p(z)|^2)^\alpha}\\
&\leq& (\frac 34)^{\a}\|\p g'\|_{\vb}+C\|g_a\|_{\za}<\infty.\\
\end{eqnarray*}

Now we use (2) and Lemma 2.4 to conclude that
$$
\sup _{n\geq 0}(n+1)^\alpha \|g(\p'){\p}^n\|_{\vb}\asymp \sup _{z \in D}\frac{(1-|z|^2)^\beta}{(1-|\p(z)|^2)^\alpha}|g(z)\p'(z)|<\infty,$$
which shows that (4) is necessary for all case.

Conversely, suppose that $g'\in \hvb$ and (4) hold. Assume that $f\in \za$. From Lemma 2.2, it follows that
\begin{eqnarray*}
(1-|z|^2)^\beta|(\ic f)''(z)|&=&(1-|z|^2)^\beta||\p'(z)g(z)f''(\p(z))+g'(z)f'(\p(z))|\\
&\leq&(1-|z|^2)^\beta|\p'(z)g(z)f''(\p(z))|+(1-|z|^2)^\beta|g'(z)f'(\p(z))|\\
&\leq&\frac{(1-|z|^2)^\beta}{(1-|\p(z)|^2)^\alpha}|\p'(z)g(z)|\|f\|_{\za}+C(1-|z|^2)^\beta|g'(z)|\|f\|_{\za}\\
&\leq&C\|f\|_{\za},
\end{eqnarray*}
$$
|\ic (f)(0)|=0,\quad\mbox{and}\quad\quad
|(\ic f)'(0)|=|f'(\varphi(0))g(0)|\leq \|f\|_{\za}|g(\varphi(0))|,
$$
which implies that $\ic$ is bounded. This completes the proof of (i).

  Next we will prove (ii).  The necessity in condition (5) has been  proved  above.
Fixed $a\in D$ with $|a|>\frac{1}{2}$, we take the function
\begin{equation}
    k_a(z)=\frac{p(\overline{a}z)}{\overline{a}}(\log\frac{1}{1-|a|})^{-1},
    \end{equation}
for ~$z \in D$, where $$ p(z)=(z-1)((1+\log\frac{1}{1-z})^2+1).$$
Then we have ~$\sup_{\frac{1}{2}<|a|<1}\|k_a\|_{\za}\leq C$ by ~$(\mbox{\rm\cite{li3}})$~.
Let ~$a=\p(z)$, it follows that
\begin{eqnarray*}
\|\ic (k_{\p(z)})\|_{\zb}&\geq& (1-|z|^2)^\beta|g'(z)||k'_{\p(z)}(\p(z))|-(1-|z|^2)^\beta|\p'(z)g(z)||k''_{\p(z)}(\p(z))|\\
&=&(1-|z|^2)^\beta|g'(z)|\log \frac{1}{1-|\p(z)|^2}-(1-|z|^2)^\beta|\p'(z)g(z)||\frac{\alpha}{(1-|\p(z)|^2)^\alpha}|.\\
\end{eqnarray*}
Since (5) holds and $\ic$ is bounded, we obtain that
\begin{eqnarray*}
\displaystyle\sup_{|\p(z)|>\frac{1}{2}} (1-|z|^2)^\beta\log \frac{1}{1-|\p(z)|^2}|g'(z)|&\leq&\displaystyle\sup_{|\p(z)|>\frac{1}{2}}(1-|z|^2)^\beta|\p'(z)g(z)||\frac{\alpha}{(1-|\p(z)|^2)^\alpha}|\\
&+&\displaystyle\sup_{|\p(z)|>\frac{1}{2}}\|\ic (k_{\p(z)})\|_{\zb}\\
&<& \infty.\\
\end{eqnarray*}
Noting ~$g'\in \hvb$, and together with (3) and Lemma 2.4, we conclude that  (6) holds.

The converse implication can be shown as in the proof of (i).

Finally we will prove (iii).  We have proved that (7) holds above. To prove (8), we take function ~$f_{\p(z)}$ defined in (9) for every ~$z\in D$ with ~$|\p(z)|>\frac{1}{2}$, and obtain that
\begin{eqnarray*}
\|\ic (f_{\p(z)})\|_{\zb}&\geq& (1-|z|^2)^\beta|g'(z)||f'_{\p(z)}(\p(z))|-(1-|z|^2)^\beta|\p'(z)g(z)||f''_{\p(z)}(\p(z))|\\
&=&(1-|z|^2)^\beta|g'(z)| \frac{1}{|\p(z)|(1-|\p(z)|^2)^{\alpha-1}}\\
&-&(1-|z|^2)^\beta|\p'(z)g(z)||\frac{2\alpha}{(1-|\p(z)|^2)^\alpha}|.\\
\end{eqnarray*}
Since ~$\ic$ is bounded and (7) holds, we obtain that
\begin{eqnarray*}
\displaystyle \sup _{|\p(z)|>\frac{1}{2}}\frac{(1-|z|^2)^\beta}{(1-|\p(z)|^2)^{\alpha-1}}|g'(z)|&\leq&\displaystyle \sup _{|\p(z)|>\frac{1}{2}}(1-|z|^2)^\beta|\p'(z)g(z)||\frac{2\alpha}{(1-|\p(z)|^2)^\alpha}|\\
&+&\displaystyle \sup _{|\p(z)|>\frac{1}{2}}\|\ic (f_{\p(z)})\|_{\zb}< \infty,\\
\end{eqnarray*}
therefore, we deduce that (8) holds by (2) and Lemma 2.4.

The converse implication can be shown as in the proof of (i).
  \end{proof}
\begin{theorem}
   Let $\varphi$ be an analytic self-map of $D$  and $g\in H(D)$.

   (i)~~If $0<\alpha<1$, then  $\ci:\za \rightarrow \zb $ is a bounded operator if and only if  ~$((g\circ \p)(\p'')+(g'\circ\p)(\p')^2) \in \hvb$ and
   \begin{equation}
   \sup _{n\geq 0}(n+1)^\alpha \|(g\circ \p)(\p')^2{\p}^n\|_{\vb}\asymp \sup _{z \in D}\frac{(1-|z|^2)^\beta}{(1-|\p(z)|^2)^\alpha}|g(\p(z))(\p'(z))^2|<\infty.
   \end{equation}
   (ii)~~If $\alpha=1$, then  $\ci:\za \rightarrow \zb $ is a bounded operator if and only if
    \begin{equation}
    \sup _{n\geq 0}(n+1) \|(g\circ \p)(\p')^2{\p}^n\|_{\vb}\asymp \sup _{z \in D}\frac{(1-|z|^2)^\beta}{1-|\p(z)|^2}|g(\p(z))(\p'(z))^2|<\infty
    \end{equation}
    and
    \begin{equation}
    \begin{aligned}
    \sup _{n\geq 0}(\log n) \|((g'\circ \p)(\p')^2+(g\circ \p)(\p'')){\p}^n\|_{\vb}\asymp \sup _{z \in D}(1-|z|^2)^\beta \log \frac{2}{1-|\p(z)|^2}\\\times|g'(\p(z))(\p'(z))^2
    +(g(\p(z))\p''(z)|<\infty.
\end{aligned}
\end{equation}

(iii)~~If $\alpha>1$, then   $\ci:\za \rightarrow \zb $ is a bounded operator if and only if (13) holds and
\begin{equation}
\begin{aligned}
 \sup _{n\geq 0}(n+1)^{\alpha-1} \|((g'\circ \p)(\p')^2+(g\circ \p)(\p'')){\p}^n\|_{\vb}\asymp \sup _{z \in D} \frac{(1-|z|^2)^\beta}{(1-|\p(z)|^2)^{\alpha-1}}\\\times|g'(\p(z))(\p'(z))^2
+(g(\p(z))\p''(z)|<\infty.
\end{aligned}
\end{equation}

  \end{theorem}

The proof is similar to that of Theorem 2.1, the details are omitted.

  \begin{theorem}
   Let $\varphi$ be an analytic self-map of $D$  and~$g\in H(D)$.

   (i)~~If $0<\alpha<1$, then $\cj:\za \rightarrow \zb $ is a bounded operator if and only if  $((g'\circ \p)(\p'')+(g''\circ\p)(\p')^2) \in \hvb$ and
   ~$(g'\circ \p)(\p')^2 \in \hvb$~.

   (ii)~~If $\alpha=1$, then $\cj:\za \rightarrow \zb $ is a bounded operator if and only if $((g'\circ \p)(\p'')+(g''\circ\p)(\p')^2) \in \hvb$ and
    \begin{equation}
    \sup _{n\geq 0}(\log n) \|(g'\circ \p)(\p')^2{\p}^n\|_{\vb}\asymp \sup _{z \in D}(1-|z|^2)^\beta\log(\frac{2}{1-|\p(z)|^2})|g'(\p(z))(\p'(z))^2|<\infty,
    \end{equation}

(iii)~~If $1<\alpha<2$, then $\cj:\za \rightarrow \zb $ is a bounded operator if and only if $((g'\circ \p)(\p'')+(g''\circ\p)(\p')^2) \in \hvb$ and
\begin{equation}
\begin{aligned}
 \sup _{n\geq 0}(n+1)^{\alpha-1} \|((g'\circ \p)(\p')^2){\p}^n\|_{\vb}\asymp \sup _{z \in D} \frac{(1-|z|^2)^\beta}{(1-|\p(z)|^2)^{\alpha-1}}\\\times|g'(\p(z))(\p'(z))^2|<\infty.
\end{aligned}
\end{equation}

(iv)~~If $\alpha=2$, then $\cj:\za \rightarrow \zb $ is a bounded operator if and only if
\begin{equation}
\begin{aligned}
 \sup _{n\geq 0}(n+1)\|((g'\circ \p)(\p')^2){\p}^n\|_{\vb}\asymp \sup _{z \in D} \frac{(1-|z|^2)^\beta}{(1-|\p(z)|^2)}\\\times|g'(\p(z))(\p'(z))^2|<\infty
\end{aligned}
\end{equation}
and
\begin{equation}
\begin{aligned}
 \sup _{n\geq 0}(\log n ) \|((g''\circ \p)(\p')^2+(g'\circ\p)\p''){\p}^n\|_{\vb}\asymp \sup _{z \in D}(1-|z|^2)^\beta \log (\frac{2}{(1-|\p(z)|^2)})\\\times|g''(\p(z))(\p'(z))^2+g'(\p(z))\p''(z)|<\infty.
\end{aligned}
\end{equation}

(v)~~If $\alpha >2$, then $\cj:\za \rightarrow \zb $ is a bounded operator if and only if (18) holds and
\begin{equation}
\begin{aligned}
 \sup _{n\geq 0}(n+1)^{\alpha-2} \|((g''\circ \p)(\p')^2+(g'\circ\p)\p''){\p}^n\|_{\vb}\asymp \sup _{z \in D} \frac{(1-|z|^2)^\beta}{(1-|\p(z)|^2)^{\alpha-2}}\\\times|g''(\p(z))(\p'(z))^2+g'(\p(z))\p''(z)|<\infty.
\end{aligned}
\end{equation}

  \end{theorem}

\begin{proof}
Suppose that $\cj$ is bounded from $\za$ to $\zb$.

(i) {\bf Case $0<\alpha<1$.} Using functions $f=1 \in \za$ and $f=z \in \za$, we obtain
\begin{equation}
 \sup _{z \in D} (1-|z|^2)^\beta|g''(\p(z))(\p'(z))^2+g'(\p(z))\p''(z)|<\infty,
\end{equation}
and
\begin{equation}
 \sup _{z \in D} (1-|z|^2)^\beta|g'(\p(z))(\p'(z))^2|<\infty.
\end{equation}
Then we obtain that $((g'\circ \p)(\p'')+(g''\circ\p)(\p')^2) \in \hvb$ and
   $(g'\circ \p)(\p')^2 \in \hvb$ are necessary for all case.

For the converse implication, suppose that $((g'\circ \p)(\p'')+(g''\circ\p)(\p')^2) \in \hvb$ and $(g'\circ \p)(\p')^2 \in \hvb$.
  For $f\in \za$,  it follows from Lemma 2.2 that
\begin{eqnarray*}
(1-|z|^2)^\beta|(\cj f)''(z)|&=&(1-|z|^2)^\beta||(\p'(z))^2g'(\p(z))f'(\p(z))+(g''(\p(z))(\p'(z))^2\\
&+&g'(\p(z))\p''(z))f(\p(z))|
\leq(1-|z|^2)^\beta|(\p'(z))^2g'(\p(z))f'(\p(z))|\\
&+&(1-|z|^2)^\beta|(g''(\p(z))(\p'(z))^2+g'(\p(z))\p''(z))f(\p(z))|\\
&\leq&(1-|z|^2)^\beta|(\p'(z))^2g'(\p(z))|\|f\|_{\za}\\
&+&(1-|z|^2)^\beta|g''(\p(z))(\p'(z))^2+g'(\p(z))\p''(z)|\|f\|_{\za}
\leq C\|f\|_{\za},\\
\end{eqnarray*}
\begin{eqnarray*}
|\cj (f)(0)|&=&|\int_0 ^{\varphi(0)} f(\zeta)g'(\zeta) d\zeta|\\
\ &\ \\
&\leq&\displaystyle\max_{\zeta \leq |\varphi(0)|}|f(\zeta)|\displaystyle\max_{\zeta \leq |\varphi(0)|}|g'(\zeta)|\leq\displaystyle\max_{\zeta \leq |\varphi(0)|}\|f\|_{\za}\displaystyle\max_{\zeta \leq |\varphi(0)|}|g'(\zeta)|,\\
\end{eqnarray*}
and
$$
|(\cj f)'(0)|=|f(\varphi(0))\varphi'(0)g'(\varphi(0))|\leq \|f\|_{\za}|\varphi'(0)||g'(\varphi(0))|.
$$
Then $\cj$ is bounded. This completes the proof of (i).

 (ii) {\bf Case $\alpha=1$.}  We consider the test function ~$ k_{\p(z)}(z)$  defined in (12) for every $z \in D$ with $|\p(z)|>\frac{1}{2}$. It follows that
 \begin{eqnarray*}
\|\cj k_{\p(z)}\|_{\zb} &\geq& \displaystyle\sup_{z \in D}(1-|z|^2)^\beta|(\cj k_{\p(z)})''(z)|\\
&\geq& \displaystyle\sup_{|\p(z)|>\frac{1}{2}}\log(\frac{1}{1-|\p(z)|^2})(1-|z|^2)^\beta|(\p'(z))^2g'(\p(z))|\\
&-& \displaystyle\sup_{|\p(z)|>\frac{1}{2}}(1-|z|^2)^\beta|\p''(z)g'(\p(z))+g''(\p(z))(\p(z))^2||k_{\p(z)}(\p(z))|.\\
\end{eqnarray*}
Since ~$((g'\circ \p)(\p'')+(g''\circ\p)(\p')^2) \in \hvb$, and $\displaystyle\sup_{|\p(z)|>\frac{1}{2}}\|\cj k_{\p(z)}\|_{\z_1}\leq C$, we get
 \begin{eqnarray*}
\ &\displaystyle\sup_{|\p(z)|>\frac{1}{2}}\log\frac{1}{1-|\p(z)|^2}(1-|z|^2)^\beta|(\p'(z))^2g'(\p(z))|  \\
\  &\displaystyle\leq\displaystyle\sup_{|\p(z)|>\frac{1}{2}}(1-|z|^2)^\beta|\p''(z)g'(\p(z))+g''(\p(z))(\p(z))^2||k_{\p(z)}(\p(z))|
 +\|\cj k_{\p(z)}\|< \infty.
\end{eqnarray*}
Then we have
\begin{eqnarray*}
 & &\displaystyle\sup_{z\in D}\log(\frac{1}{1-|\p(z)|^2})(1-|z|^2)^\beta|(\p'(z))^2g'(\p(z))|\\
&\leq& \displaystyle\sup_{|\p(z)|\leq\frac{1}{2}}\log(\frac{1}{1-|\p(z)|^2})(1-|z|^2)^\beta|(\p'(z))^2g'(\p(z))|\\
&+&  \displaystyle\sup_{|\p(z)|>\frac{1}{2}}\log(\frac{1}{1-|\p(z)|^2})(1-|z|^2)^\beta|(\p'(z))^2g'(\p(z))| \\
  &\leq& C+\log(\frac 43)\|(g'\circ \p)(\p')^2\|_{\vb}<\infty.
\end{eqnarray*}
On the other hand, from  (3) and Lemma 2.4, we have
$$
    \sup _{n\geq 0}(\log n) \|(g'\circ \p)(\p')^2{\p}^n\|_{\vb}\asymp \sup _{z \in D}(1-|z|^2)^\beta\log(\frac{2}{1-|\p(z)|^2})|g'(\p(z))(\p'(z))^2|.
 $$
Hence (17) holds.

  The converse implication can be shown as in the proof of (i).

(iii) {\bf Case $1<\alpha<2$.} $((g'\circ \p)(\p'')+(g''\circ\p)(\p')^2) \in \hvb$ has been proved above.  We take the test function $f_{\p(z)}$ in (9) for every $z \in D$ with $|\p(z)|>\frac{1}{2}$ , by the same way as (ii), we can obtain that (18) holds.

The converse implication can be shown as in the proof of (i).

(iv) {\bf Case $\alpha=2$.} We have prove that (19) holds above. To prove (20), we consider another test function ~$t_a(z)=\log\frac{2}{1-\overline{a}z}$.  Clearly ~$t_a\in \zz$  and $\sup_{\frac{1}{2}<|a|<1}\|t_a\|_{\zz}<\infty$. For every ~$z \in D$ with ~$|\p(z)|>\frac{1}{2}$, it follows that
 \begin{eqnarray*}
\displaystyle\sup_{z \in D}(1-|z|^2)^\beta|(\cj t_a)''(z)|
&\geq& \displaystyle\sup_{|\p(z)|>\frac{1}{2}}\log(\frac{1}{1-|\p(z)|^2})(1-|z|^2)^\beta|(\p'(z))^2g''(\p(z))\\
&+&g'(\p(z))\p''(z)|- \displaystyle\sup_{|\p(z)|>\frac{1}{2}}\frac{(1-|z|^2)^\beta}{1-|\p(z)|^2}|g'(\p(z))(\p(z))^2|.\\
\end{eqnarray*}
Applying (19) we get
\begin{eqnarray*}
&\ &\displaystyle\sup_{|\p(z)|>\frac{1}{2}}\log(\frac{1}{1-|\p(z)|^2})(1-|z|^2)^\beta|(\p'(z))^2g''(\p(z))+g'(\p(z))\p''(z)|\\
&\leq& \displaystyle\sup_{|\p(z)|>\frac{1}{2}}\frac{(1-|z|^2)^\beta}{1-|\p(z)|^2}|g'(\p(z))(\p(z))^2|+\|\cj t_a\| < \infty.\\
\end{eqnarray*}

Noting $((g'\circ \p)(\p'')+(g''\circ\p)(\p')^2) \in \hvb$, and using Lemma 2.4 and (3), we conclude that (20) holds.

(v) {\bf Case $\alpha>2$.}  We have  proved that  (18) holds above.
Applying test function ~$f_{\p(z)}$ in (9) for every ~$z \in D$  with $|\p(z)|>\frac{1}{2}$, we have
\begin{eqnarray*}
S_1=\displaystyle\sup_{|\p(z)|>\frac{1}{2}}(1-|z|^2)^\beta|\frac{(\p'(z))^2g'(\p(z))}{\overline{\p(z)}(1-|\p(z)|^2)^{\alpha-1}}|
&=&\displaystyle\sup_{|\p(z)|>\frac{1}{2}}(1-|z|^2)^\beta|(\p'(z))^2g'(\p(z))||f'_{\p(z)}(\p(z))|\\
&\leq& \displaystyle\sup_{|\p(z)|>\frac{1}{2}}\|\cj(f_{\p(z)})\|_{\zb}<\infty.\\
\end{eqnarray*}
With the same calculation for test function $t_{\p(z)}(\p(z))=\frac{(1-|\p(z)|^2)^2}{(1-\overline{\p(z)}z)^\alpha}$ with $|\p(z)|>\frac{1}{2}$, then $\sup_{|\p(z)|>\frac{1}{2}}\|t_{\p(z)}\|_{\za}\leq C$, and we have that
\begin{eqnarray*}
S_2&=&\displaystyle\sup_{|\p(z)|>\frac{1}{2}}(1-|z|^2)^\beta|\frac{(\p'(z))^2g''(\p(z))+g'(\p(z))\p''(z)}{(1-|\p(z)|^2)^{\alpha-2}}
+\alpha\overline{\p(z)}\frac{(\p'(z))^2g'(\p(z))}{(1-|\p(z)|^2)^{\alpha-1}}|\\
&=&\displaystyle\sup_{|\p(z)|>\frac{1}{2}}(1-|z|^2)^\beta|(\cj(t_{\p(z)}))''(\p(z))|
\leq \|\cj\|_{\za\rightarrow \zb}\displaystyle\sup_{|\p(z)|>\frac{1}{2}}\|t_{\p(z)}\|_{\za}<\infty.\\
\end{eqnarray*}
Therefore,
\begin{eqnarray*}
\ &\displaystyle\sup_{|\p(z)|>\frac{1}{2}}(1-|z|^2)^\beta|\frac{(\p'(z))^2g''(\p(z))+g'(\p(z))\p''(z)}{(1-|\p(z)|^2)^{\alpha-2}}|\\
&\leq  S_2+\alpha\displaystyle\sup_{|\p(z)|>\frac{1}{2}}\frac{|(\p'(z))^2g'(\p(z))|}{(1-|\p(z)|^2)^{\alpha-1}}\\
&\leq S_2+\alpha\displaystyle\sup_{|\p(z)|>\frac{1}{2}}\frac{|(\p'(z))^2g'(\p(z))|}{|\overline{\p(z)}|(1-|\p(z)|^2)^{\alpha-1}}\leq  S_2+\alpha S_1<\infty.\\
\end{eqnarray*}
Since $((g'\circ \p)(\p'')+(g''\circ\p)(\p')^2) \in \hvb$, we conclude that (21) holds.

\end{proof}

  \begin{theorem}
   Let $\varphi$ be an analytic self-map of $D$  and~$g\in H(D)$.

   (i)~~If ~$0<\alpha<1$, then $\jc:\za \rightarrow \zb $~is a bounded operator if and only if~$g'\p' \in \hvb$ and~$g'' \in \hvb$.

   (ii)~~If ~$\alpha=1$, then $\jc:\za \rightarrow \zb $~is a bounded operator if and only if~$g'' \in \hvb$ and
    \begin{equation}
    \sup _{n\geq 0}(\log n) \|(g'\p'){\p}^n\|_{\vb}\asymp \sup _{z \in D}(1-|z|^2)^\beta\log \frac{2}{1-|\p(z)|^2}|g'(z)\p'(z)|<\infty.
    \end{equation}

(iii)~~If ~$1<\alpha<2$, then $\jc:\za \rightarrow \zb $~is a bounded operator if and only if ~$g'' \in \hvb$ and
\begin{equation}
 \sup _{n\geq 0}(n+1)^{\alpha-1} \|(g'\p'){\p}^n\|_{\vb}\asymp \sup _{z \in D} \frac{(1-|z|^2)^\beta}{(1-|\p(z)|^2)^{\alpha-1}}|g'(z)\p'(z)|<\infty.
\end{equation}

(iv)~~If ~$\alpha=2$, then $\jc:\za \rightarrow \zb $~is a bounded operator if and only if

\begin{equation}
 \sup _{n\geq 0}(n+1) \|(g'\p'){\p}^n\|_{\vb}\asymp \sup _{z \in D} \frac{(1-|z|^2)^\beta}{(1-|\p(z)|^2)}|g'(z)\p'(z)|<\infty,
\end{equation}
and
\begin{equation}
 \sup _{n\geq 0}(\log n ) \|(g''){\p}^n\|_{\vb}\asymp \sup _{z \in D}(1-|z|^2)^\beta \log \frac{2}{(1-|\p(z)|^2)}|g''(z)|<\infty.
\end{equation}
(v)~~If ~$\alpha >2$, then $\jc:\za \rightarrow \zb $~is a bounded operator if and only if (25) holds and
\begin{equation}
 \sup _{n\geq 0}(n+1)^{\alpha-2} \|(g''){\p}^n\|_{\vb}\asymp \sup _{z \in D} \frac{(1-|z|^2)^\beta}{(1-|\p(z)|^2)^{\alpha-2}}|g''(z)|<\infty.
\end{equation}
  \end{theorem}
The proof is similar to that of Theorem 2.3, the details are omitted.
\section{Essential norms}

  In this section we estimate  the essential norms  of these integral-type  operators on~$Zygmund$~ type spaces in terms of ~$g,\varphi$, their derivatives and the n-th power ~$\varphi^n$ of ~$\varphi$.

   Let $\widetilde{\za}=\{f\in\za :f(0)=f'(0)=0\}$ and ~$\widetilde{\ba}=\{f\in \ba :f(0)=0\}$.
  We note that every compact operator ~$T\in \mathbf{K}(\widetilde{\za},\zb) $ can be extended to a compact operator $K \in \mathbf{K}(\za,\zb)$. In fact, for every ~$f\in \za$,~$f-f(0)-f'(0)z \in \widetilde{\za}$, and we can define ~$K(f)=T(f-f(0)-f'(0)z)+f(0)+f'(0))z$.

For $r \in (0,1)$, we consider the compact operator ~$K_r:\za\rightarrow\zb$~ defined by ~$K_rf(z)=f(rz)$.
\begin{lemma}
If ~$X(\ic,\ci,\cj,\jc)$~ is a bounded operator from $\za$ to $\zb$ space, then $$\|X\|_{e,\za\rightarrow\zb}= \|X\|_{e,\widetilde{\za}\rightarrow\zb}.$$
\end{lemma}

\begin{proof}
Clearly $\|X\|_{e,\za\rightarrow\zb}\geq \|X\|_{e,\widetilde{\za}\rightarrow\zb}$.
Then we prove~$\|X\|_{e,\za\rightarrow\zb}\leq\|X\|_{e,\widetilde{\za}\rightarrow\zb}$.

Let~$T\in K(\za,\zb)$ be given. Let~$\{r_n\}$~be an increasing sequence in ~$(0,1)$ converging to ~$1$ and $\mathbf{A}=\{h|h=a+bz\}$, the closed subspace of ~$\za$. Then
\begin{eqnarray*}
\|X-T\|_{\za\rightarrow\zb} &=&\sup_{f \in \za,\|f\|_{\za}\leq1}\|X(f)-T(f)\|_{\za\rightarrow\zb}\\
&\leq&\sup_{f \in \za,\|f\|_{\za}\leq1}\|X(f-f(0)-f'(0)z)-T(f-f(0)-f'(0)z)\|_{\za\rightarrow\zb}\\
&+&\sup_{f \in \za,\|f\|_{\za}\leq1}\|X(f(0)+f'(0)z)-T(f(0)+f'(0)z)\|_{\za\rightarrow\zb}\\
&\leq&\sup_{g \in \widetilde{\za},\|g\|_{\za}\leq1}\|X(g)-T|_{\widetilde{\za}}(g)\|_{\za\rightarrow\zb}+\sup_{h \in \a,\|h\|_{\a}\leq1}\|X(h)-T|_{\a}(h)\|_{\za\rightarrow\zb}.\\
\end{eqnarray*}
Hence
\begin{eqnarray*}
\inf_{T\in K(\za,\zb)}\|X-T\|_{\za\rightarrow\zb} &\leq&\inf_{T\in K(\za,\zb)}\|X-T|_{\widetilde{\za}}\|_{\za\rightarrow\zb}
+\inf_{T\in K(\za,\zb)}\|X-T|_{\a}\|_{\za\rightarrow\zb}\\
&\leq&\|X\|_{e,\widetilde{\za}\rightarrow\zb}+\displaystyle\lim_{n\rightarrow\infty}\|X(I-K_{r_n})\|_{\a\rightarrow\zb}.\\
\end{eqnarray*}
Since $\displaystyle\lim_{n\rightarrow\infty}\|X(I-K_{r_n})\|_{\a\rightarrow\zb}=0$,
we have, $\|X\|_{e,\za\rightarrow\zb}\leq\|X\|_{e,\widetilde{\za}\rightarrow\zb}$,
and the proof is finished .

\end{proof}

Let ~$D_\alpha:\za\rightarrow \B_\alpha$~ and ~$\sa:\ba\rightarrow \hva$~ be the derivative operators. Then clearly ~$\da$~and~$\sa$ are linear isometries on ~$\widetilde{\za}$ and ~$\widetilde{\ba}$,  respectively and~$$S_\beta D_\beta\ic \da^{-1}\sa^{-1}=g(\p')\cp+g'\cp \sa^{-1},$$
Therefore
\begin{equation}
\|\ic\|_{e,\widetilde{\za}\rightarrow \zb}\leq \|g'\cp\|_{e,\widetilde{\ba}\rightarrow \hvb}+\|g(\p')\cp\|_{e,\hva\rightarrow \hvb}.
\end{equation}
Similarly, we have

 $\sb\db\ci \da^{-1}\sa^{-1}=(g\circ\p(\p')^2)\cp+(g'\circ\p(\p')^2+g\circ\p(\p''))\cp \sa^{-1},$
\begin{equation}
\|\ci\|_{e,\widetilde{\za}\rightarrow \zb}\leq \|((g'\circ\p(\p')^2))\cp\|_{e,\hva\rightarrow \hvb}+\|((g'\circ\p)(\p')^2+g\circ\p(\p''))\cp\|_{e,\widetilde{\ba}\rightarrow \hvb}.
\end{equation}
$\sb\db\cj \da^{-1}\sa^{-1}=(g''\circ\p(\p')^2+g'\circ\p(\p''))\cp \da^{-1}\sa^{-1}+(g'\circ\p(\p')^2)\cp \sa^{-1},$
\begin{equation}
\|\cj\|_{e,\widetilde{\za}\rightarrow \zb}\leq \|((g'\circ\p)(\p')^2)\cp\|_{e,\widetilde{\ba}\rightarrow \hvb}+\|((g''\circ\p)(\p')^2+g'\circ\p(\p''))\cp\|_{e,\widetilde{\za}\rightarrow \hvb}.
\end{equation}
$\sb\db\jc \da^{-1}\sa^{-1}=g''\cp \da^{-1}\sa^{-1}+g'(\p')\cp \sa^{-1},$
\begin{equation}
\|\jc\|_{e,\widetilde{\za}\rightarrow \zb}\leq \|g'(\p')\|_{e,\widetilde{\ba}\rightarrow \hvb}+\|g''\cp\|_{e,\widetilde{\za}\rightarrow \hvb}.
\end{equation}
\begin{theorem}
  Let $\varphi$ be an analytic self-map of $D$  and~$g\in H(D)$, and $\ic:\za \rightarrow \zb $~is a bounded operator.

  (i)~~If ~$0<\alpha<1$, then
  \begin{equation}
\|\ic\|_{e,\za\rightarrow \zb}\asymp \displaystyle\limsup_{n\rightarrow\infty}(n+1)^\alpha\|(g\p'){\p}^n\|_{\vb}.
\end{equation}

 (ii)~~If ~$\alpha=1$, then
  \begin{equation}
\|\ic\|_{e,\z\rightarrow \zb}\asymp \max\{\displaystyle\limsup_{n\rightarrow\infty}(n+1)\|(g\p'){\p}^n\|_{\vb},\displaystyle\limsup_{n\rightarrow\infty}(\log n)\|(g'){\p}^n\|_{\vb}\}.
\end{equation}

(iii)~~If ~$\alpha>1$, then
  \begin{equation}
\|\ic\|_{e,\za\rightarrow \zb}\asymp \max\{\displaystyle\limsup_{n\rightarrow\infty}(n+1)^\alpha\|(g\p'){\p}^n\|_{\vb},\displaystyle\limsup_{n\rightarrow\infty}( n+1)^{\alpha-1}\|(g'){\p}^n\|_{\vb}\}.
\end{equation}

  \end{theorem}

\begin{proof}
(i) We start with the upper bound. First we show that $g'C_{\p}$ is a compact weighted composition operator for $\ba$ into $\hvb$.
Suppose that ~$\{f_n\}$~ is bounded sequence in ~$\ba$. From Lemma 3.6 in \cite{osk},  $\{f_n\}$ has a subsequence ~$\{f_{n_k}\}$  which converges uniformly on ~$D$ ~to a function, which we can assume to be  identically zero. Then it follows  from Theorem 2.1 and Lemma 2.1 that
 ~$$\displaystyle\lim_{k\rightarrow\infty}\sup_{z \in D}(1-|z|^2)^\beta|g'(z)||f_{n_k}(\p(z))|\leq C \displaystyle\lim_{k\rightarrow\infty}\sup_{z \in D}|f_{n_k}(z)|=0,$$~
which shows that ~$g'C_{\p}: \ba\to \hvb$~ is a compact operator and ~$\|g'\cp\|_{e,\ba\rightarrow \hvb}=0$.
Applying (29), Lemma 2.5 and 3.1, we get that
 \begin{eqnarray*}
 \|\ic\|_{e,\za\rightarrow \zb}\leq \|g(\p')\cp\|_{e,\hva\rightarrow \hvb}&=&\displaystyle\limsup_{|\p(z)|\rightarrow 1}\frac{(1-|z|^2)^\beta}{(1-|\p(z)|^2)^\alpha}|g(z)\p'(z)|\\
&=& \displaystyle\limsup_{n\rightarrow\infty}\frac{\|(g\p'){\p}^n \|_{\vb}}{\|z^n\|_{\va}}\\
&=&(\frac{e}{2\alpha})^\alpha\displaystyle\limsup_{n\rightarrow\infty}(n+1)^\alpha\|(g\p'){\p}^n\|_{\vb}.\\
\end{eqnarray*}

For the lower bound, let ~$\{z_n\}\subseteq D$  with ~$|\p(z_n)|>\frac{1}{2}$ and ~$|\p(z_n)|\rightarrow1$~ as ~$n\rightarrow\infty.$ Taking ~$g_n=g_{\p(z_n)}$~ defined in (11),  we obtain that ~$\{g_n\}$~ is bounded sequence in ~$\zao$~ converging to ~$0$~ uniformly on compact subset of ~$D$~ and ~$\displaystyle\sup_{n \in N}\|g_n\|_{\za}\leq C$. For every compact operator ~$T:\za\rightarrow \zb$~,
\begin{eqnarray*}
C\|\ic-T\|_{\za\rightarrow \zb}&\geq&\displaystyle\limsup_{n\rightarrow\infty}\|\ic g_n\|_{\zb}-\displaystyle\limsup_{n\rightarrow\infty}\|T(g_n)\|_{\zb}\\
&\geq&\alpha\displaystyle\limsup_{n\rightarrow\infty}(1-|z_n|^2)^\beta\frac{|g(z_n)\p'(z_n)|}{(1-|\p(z_n)|^2)^\alpha}.\\
\end{eqnarray*}
Now we   use (2) and Lemma 2.4 to obtain that
\begin{eqnarray*}
\|\ic\|_{e,\za\rightarrow \zb}&\geq&\|\ic-T\|_{\za\rightarrow \zb}
\geq\frac{\alpha}{C}\displaystyle\limsup_{|\p(z_n)|\rightarrow 1}|g(z_n)\p'(z_n)|\frac{(1-|z_n|^2)^\beta}{(1-|\p(z_n)|^2)^\alpha}\\
&=&\frac{\alpha}{C}\displaystyle\limsup_{n\rightarrow\infty}\frac{\|(g\p'){\p}^n\|_{\vb}}{\|z^n\|_{\va}}\\
&=&C(\frac{e}{2\alpha})^\alpha\displaystyle\limsup_{n\rightarrow\infty}(n+1)^\alpha\|(g\p'){\p}^n\|_{\vb}.\\
\end{eqnarray*}
Hence (33) holds.

(ii)  The boundedness of $\ic $ guarantees that $(g\p')\cp:\hv\rightarrow\hvb$ and $(g')\cp:\B\rightarrow\hvb$
are bounded weighted composition operators.
Theorem  3.4 in $\cite{ste5}$ ensures that  $$\|g'\cp\|_{e,\B\rightarrow \hvb}\asymp \displaystyle\lim_{|\p(z)|\rightarrow1}(1-|z|^2)^\beta|g'(z)|\log\frac{2}{1-|\p(z)|^2}.$$
  Now we use  Lemma 2.4, 2.5 and (29) to conclude that
\begin{eqnarray*}
\|\ic\|_{e,\z\rightarrow \zb} &\leq& \|g'\cp\|_{e,\widetilde{\B}\rightarrow \hvb}+\|g(\p')\cp\|_{e,\hv\rightarrow \hvb}\\
&\leq& C\displaystyle\limsup_{n\rightarrow\infty}\frac{\|(g\p'){\p}^n \|_{\vb}}{\|z^n\|_{v_1}}+C\displaystyle\limsup_{n\rightarrow\infty}\frac{\|g'{\p}^n \|_{\vb}}{\|z^n\|_{v_{\log}}}\\
&=& C\displaystyle\limsup_{n\rightarrow\infty}(\log n)\|g'{\p}^n \|_{\vb}+\frac{Ce}{2}\displaystyle\limsup_{n\rightarrow\infty}( n+1)\|(g\p'){\p}^n \|_{\vb}\\
&\leq&C\max\{\displaystyle\limsup_{n\rightarrow\infty}(n+1)\|(g\p'){\p}^n\|_{\vb},\displaystyle\limsup_{n\rightarrow\infty}(\log n)\|(g'){\p}^n\|_{\vb}\}.
\end{eqnarray*}

  On the other hand, let ~$\{z_n\}$ be a sequence in ~$ D$ such that ~$|\p(z_n)|>\frac{1}{2}$ and $|\p(z_n)|\rightarrow 1$ as $n\rightarrow\infty$.
 Given
 \begin{equation}
 h_n(z)=\frac{h(\overline{\p(z_n)}z)}{\overline{\p(z_n)}}(\log\frac{2}{1-|\p(z_n)|^2})^{-1}-(\int_0^z\log^3\frac{2}{1-\overline{\p(z_n)}\omega}d\omega)
 (\log\frac{2}{1-|\p(z_n)|^2})^{-2},
\end{equation}
where ~$h(z)=(z-1)((1+\log\frac{2}{1-z})^2+1)$. From \cite{li3} we know that
~$\{h_n\}$ is a bounded sequence in ~$\z_1^0$ which converges to zero uniformly on compact subsets of ~$D$, and
$$h_n''(\p(z_n))=\frac{-\overline{\p(z_n)}}{1-|\p(z_n)|^2}, ~~~~~~h_n'(\p(z_n))=0, ~~~~~~\sup_n\|h_n\|_{\z}<+\infty.~~~~~$$
For every compact operator~$T:\z\rightarrow \zb$~, we have $\|T(h_n)\|_{\zb}\rightarrow0$ as ~$n\rightarrow\infty$.  By Lemma 2.4 and 2.5, we obtain that
\begin{eqnarray*}
C\|\ic-T\|_{\z\rightarrow\zb}&\geq&\|\ic h_n\|_{\z\rightarrow\zb}
\geq\displaystyle\limsup_{n\rightarrow\infty}(1-|z_n|^2)^\beta||g(z_n)\p'(z_n)|\frac{|\p(z_n)|}{1-|\p(z_n)|^2}\\
&\geq&\displaystyle\limsup_{n\rightarrow\infty}\frac{\|(g\p'){\p}^n\|_{\vb}}{\|z^n\|_{v_1}}\\
&=&\frac{e}{2}\displaystyle\limsup_{n\rightarrow\infty}(n+1)\|(g\p'){\p}^n\|_{\vb}.\\
\end{eqnarray*}
Now we take another function
\begin{equation}
~f_n=\frac{h(\overline{\p(z_n)}z)}{\overline{\p(z_n)}}(\log\frac{2}{1-|\p(z_n)|^2})^{-1}.
\end{equation}
From \cite{li3} we know that
$\{f_n\}$~ is a bounded sequence in ~$\z_1^0$ which converges to zero uniformly on compact subsets of ~$D$, and $\displaystyle\sup_{n\geq1}\|f_n\|_{\z}<+\infty$. It follows from Lemma 2.4 and 2.5 that
\begin{eqnarray*}
C\|\ic\|_{e,\z\rightarrow\zb}&\geq&\displaystyle\lim_{n\rightarrow\infty}\sup\|\ic f_n\|_{\zb}
\geq\displaystyle\limsup_{n\rightarrow\infty}(1-|z_n|^2)^\beta|g'(z_n)|\log\frac{2}{1-|\p(z_n)|^2}\\
&-&\displaystyle\limsup_{n\rightarrow\infty}(1-|z_n|^2)^\beta|(g(z_n)\p'(z_n)|\frac{|\p(z_n)|}{1-|\p(z_n)|^2}.
\end{eqnarray*}
 Noting that $\displaystyle\limsup_{n\rightarrow\infty}(n+1)\|(g\p'){\p}^n\|_{\vb}\leq \frac{2C}{e}\|\ic\|_{e,\z\rightarrow\zb}$, we obtain
\begin{eqnarray*}
(C+\frac{2C}{e})\|\ic\|_{e,\z\rightarrow\zb}&\geq&\displaystyle\limsup_{|\p(z_n)|\rightarrow 1}(1-|z_n|^2)^\beta|g'(z_n)|\log\frac{2}{1-|\p(z_n)|^2}\\
&=&\displaystyle\limsup_{n\rightarrow\infty}(\log n)\|(g'){\p}^n\|_{\vb}.
\end{eqnarray*}
Hence we have~$\|\ic\|_{e,\z\rightarrow\zb}\geq C\max\{\displaystyle\limsup_{n\rightarrow\infty}(n+1)\|(g\p'){\p}^n\|_{\vb},\displaystyle\limsup_{n\rightarrow\infty}(\log n)\|(g'){\p}^n\|_{\vb}\}.$

(iii) Let ~$\alpha>1$.  The proof of the upper bound is similar to that of (ii).
From  the proof of (i), we get that for some constant~$C$~,

\begin{eqnarray}
C\|\ic\|_{e,\za\rightarrow\zb}\geq\displaystyle\limsup_{|\p(z)|\rightarrow 1}|g(z)\p'(z)|\frac{(1-|z_n|^2)^\beta}{(1-|\p(z)|^2)^\alpha}.
\end{eqnarray}

Now, let ~$\{z_n\}$~ be as before and note that the function ~$f_n=f_{\p(z_n)}$~ given in (9). Then ~$\{f_n\}$~ is bounded sequence in ~$\zao$~ converging to zero uniformly on compact subsets of ~$D$~ , therefore
\begin{eqnarray*}
C\|\ic\|_{e,\za\rightarrow \zb}&\geq&\displaystyle\lim_{n\rightarrow \infty}\|\ic(f_n)\|_{\zb}\\
&\geq&2\alpha\displaystyle\limsup_{n\rightarrow\infty}\frac{(1-|z_n|^2)^\beta}{(1-|\p(z_n)|^2)^{\alpha-1}}|g'(z_n)|
-\displaystyle\limsup_{n\rightarrow\infty}\frac{(1-|z_n|^2)^\beta|g(z_n)\p'(z_n)|}{|\p(z_n)|(1-|\p(z_n)|^2)^\alpha}.\\
\end{eqnarray*}
By (38), we have
\begin{eqnarray*}
C\|\ic\|_{e,\za\rightarrow \zb}&\geq&
\displaystyle\limsup_{n\rightarrow\infty}\frac{(1-|z_n|^2)^\beta}{(1-|\p(z_n)|^2)^{\alpha-1}}|g'(z_n)|,
\end{eqnarray*}
and the rest of the proof is similar to that of the previous, we omit it .

\end{proof}

\begin{theorem}
 Let $\varphi$ be an analytic self-map of $D$  and~$g\in H(D)$~, and ~$\cj:\za \rightarrow \zb $~is a bounded operator.

(i)~If  $0<\alpha<1$, then
\begin{equation}
\|\cj\|_{e,\za\rightarrow \zb}=0.
\end{equation}

(ii)~If $\alpha=1$, then
\begin{equation}
\|\cj\|_{e,\z\rightarrow \zb}\asymp \displaystyle\lim_{n\rightarrow\infty}\sup(\log n)\|((g'\circ\p)(\p')^2){\p}^n\|_{\vb}.
\end{equation}

(iii)~If~$1<\alpha<2$, then
\begin{equation}
\|\cj\|_{e,\za\rightarrow \zb}\asymp \displaystyle\lim_{n\rightarrow\infty}\sup( n+1)^{\alpha-1}\|((g'\circ\p)(\p')^2){\p}^n\|_{\vb}.
\end{equation}

(iv)~If~$\alpha=2$, then
\begin{equation}
\begin{aligned}
 \|\cj\|_{e,\z^2\rightarrow \zb}\asymp \max\{\displaystyle\lim_{n\rightarrow\infty}\sup( n+1)\|((g'\circ\p)(\p')^2){\p}^n\|_{\vb},\\
\displaystyle\lim_{n\rightarrow\infty}\sup(\log  n)\|((g''\circ\p)(\p')^2+(g'\circ\p)(\p'')){\p}^n\|_{\vb}\}.
\end{aligned}
 \end{equation}

(v)~ If~$\alpha>2$, then

\begin{equation}
\begin{aligned}
 \|\cj\|_{e,\za\rightarrow \zb}\asymp  \max\{\displaystyle\lim_{n\rightarrow\infty}\sup( n+1)^{\alpha-1}\|((g'\circ\p)(\p')^2){\p}^n\|_{\vb},\\
 \displaystyle\lim_{n\rightarrow\infty}\sup( n+1)^{\alpha-2}\|((g''\circ\p)(\p')^2+(g'\circ\p)(\p'')){\p}^n\|_{\vb}\}.
\end{aligned}
 \end{equation}

\end{theorem}

\begin{proof} (i)
 For the compactness of ~$((g'\circ\p)(\p')^2)\cp$~, the arguement is similar to the proof of Theorem 3.1(i), then we have ~$\|((g'\circ\p)(\p')^2)\cp\|_{e,\ba\rightarrow \hvb}=0$. Hence by (31) and Lemma 2.3, we get that ~$\|\cj\|_{e,\za\rightarrow\zb}=0.$

Next we will prove (ii). The boundedness of ~$\cj $~ guarantees that ~$((g'\circ\p)(\p')^2)\cp:\B\rightarrow\hvb$~ and ~$((g''\circ\p)(\p')^2+(g'\circ\p)(\p''))\cp:\z\rightarrow\hvb$~
are bounded weighted composition operators.
We know that if ~$\uc:\z\rightarrow \hvb$~ is a bounded operator, then ~$\uc$~ is a compact operator by  Lemma 2.3. Hence we consider the boundedness of ~$\cj:\z\rightarrow\zb $~ just consider that ~$((g'\circ\p)(\p')^2)\cp:\B\rightarrow\hvb$~ is a bounded operator.

Theorem 3.4 in~$(\mbox{\rm\cite{ste5}})$ ensures that ~$$\|((g'\circ\p)(\p')^2)\cp\|_{e,\B\rightarrow \hvb}\asymp \displaystyle\lim_{|\p|\rightarrow1}(1-|z|^2)^\beta|g'(\p(z))(\p'(z))^2|\log\frac{2}{1-|\p(z)|^2},$$~
From (31), Lemma 2.4 and 2.5,  we have
\begin{eqnarray*}
\|\cj\|_{e,\z\rightarrow\zb} &\leq& C\displaystyle\limsup_{n\rightarrow\infty}\frac{\|((g'\circ\p)(\p')^2){\p}^n \|_{\vb}}{\|z^n\|_{v_{\log}}}\\
\  & \\
&=& C\displaystyle\limsup_{n\rightarrow\infty}(\log n)\|((g'\circ\p)(\p')^2){\p}^n \|_{\vb}\\
\end{eqnarray*}

In order to prove ~$\|\cj\|_{e,\z\rightarrow \zb}\geq \displaystyle\limsup_{n\rightarrow\infty}(\log n)\|((g'\circ\p)(\p')^2\}$,
we take the function
\begin{equation}
\displaystyle
g_n(z)=\frac{\overline{\varphi(z_n)}z-1}{\overline{\varphi(z_n)}}
\big((1+\log\frac{1}{1-\overline{\varphi(z_n)}z})^2+1\big)(\log\frac{1}{1-|\varphi(z_n)|^2})^{-1}-a_n,
\end{equation}
where  $$\displaystyle a_n=
\frac{|\varphi(z_n)|^2-1}{\overline{\varphi(z_n)}}
\big((1+\log\frac{1}{1-|\varphi(z_n)|^2})^2+1\big)(\log\frac{1}{1-|\varphi(z_n)|^2})^{-1},$$
 and $\displaystyle\lim_{n\to\infty} a_n=0$. From \cite{ysl1} we obtain that ~$\{ g_n\}$~is a bounded sequence in ~$\z_1^0$~ which converges to zero uniformly on compact subsets of ~$D$. By a directly calculation, we have
 $$g_n(\p(z_n))= 0,
~~g'_n(\p(z_n))=\log\frac{1}{1-|\p(z_n)|^2}.$$
For every compact operator ~$T:\z\rightarrow \zb$,  we have $\|T(h_n)\|_{\zb}\rightarrow0$ as ~$n\rightarrow\infty$.  Let $M=\displaystyle\sup_{n\geq1}\|g_n\|_{\zb}$. It follows from Lemma 2.5 that
\begin{eqnarray*}
M\|\cj-T\|_{e,\z\rightarrow\zb}&\geq&\|\cj f_n\|_{e,\z\rightarrow\zb}\\
&\geq&\displaystyle\limsup_{n\rightarrow\infty}(1-|z_n|^2)^\beta|g'(\p(z_n))(\p'(z_n))^2|\log\frac{1}{1-|\p(z_n)|^2}\\
&\geq&\displaystyle\limsup_{n\rightarrow\infty}\frac{\|((g'\circ\p)(\p')^2){\p}^n\|_{\vb}}{\|z^n\|_v}\\
&\geq&\displaystyle\limsup_{n\rightarrow\infty}(\log n)\|((g'\circ\p)(\p')^2){\p}^n\|_{\vb}.\\
\end{eqnarray*}
This  completes the proof.

 The proof of (iii) is the same as that of Theorem 5(iii), we don't prove it again.

(iv) Let ~$\alpha=2$. Applying Lemma 2.3 (ii) and Theorem 3.2 in \cite{ste6}, we get that
\begin{eqnarray}
\|(g'\circ\p)(\p')^2\cp\|_{e,\B^2\rightarrow \hvb}\asymp \displaystyle\limsup_{n\to\infty} (n+1)\|(g'\circ\p)(\p')^2{\p}^n\|_{\vb}
\end{eqnarray}
and
\begin{equation}
\|((g''\circ\p)(\p')^2+(g'\circ\p)\p'')\cp\|_{e,\z_2\rightarrow \hvb}\asymp \displaystyle\limsup_{n\to\infty}(\log n)\|((g''\circ\p)(\p')^2+(g'\circ\p)\p''){\p}^n\|_{\vb},
\end{equation}
which  yields the upper bound by (31).

With the same arguements as in the proof of Theorem 2.3 and 3.1, for some constant ~$C$, we have
$$
C\|\cj\|_{e,\z^2\rightarrow\zb}\geq\displaystyle\limsup_{n\rightarrow\infty}(n+1)\|((g'\circ\p)(\p')^2){\p}^n\|_{\vb}.
$$

On the other hand, let ~$\{z_n\}\subseteq D$ ~with ~$|\p(z_n)|>\frac{1}{2}$~ and ~$|\p(z_n)|\rightarrow 1$~ as ~$n\rightarrow\infty.$  Let the test function ~$$O_n(z)=(1+(\log\frac{2}{1-\overline{\p(z_n)}z})^2)(\log\frac{2}{1-|\p(z_n)|^2})^{-1}. $$
From \cite{km} we obtain that ~$\{O_n\}$~ is a bounded sequence in ~$\z^2_0$~ which converges to zero uniformly on compact subsets of $D$, and
\begin{eqnarray*}
\displaystyle\lim_{n\rightarrow\infty}(1-|z_n|^2)^\beta|g'(\p(z_n))(\p'(z_n))^2||O'_n(\p(z_n))|
&=&2\displaystyle\lim_{n\rightarrow\infty}\frac{(1-|z_n|^2)^\beta}{1-|\p(z_n)|^2}|g'(\p(z_n))(\p'(z_n))^2||\p(z_n)|\\
&\leq&C\|\cj\|_{e,\z^2\rightarrow \zb}.
\end{eqnarray*}
Applying Theorem 2.3 we get
\begin{eqnarray*}
C\|\cj\|_{e,\z_2\rightarrow\zb}&\geq&\|\cj(O_n)\|_{\zb}\\
&\geq&\displaystyle\lim_{n\rightarrow\infty}(1-|z_n|^2)^\beta|g''(\p(z_n))(\p'(z_n))^2+g'(\p(z_n))\p''(z_n)||O_n(\p(z_n))|\\
&-&\displaystyle\lim_{n\rightarrow\infty}(1-|z_n|^2)^\beta|g'(\p(z_n))(\p'(z_n))^2||O'_n(\p(z_n))|\\
&\geq&\displaystyle\lim_{n\rightarrow\infty}(1-|z_n|^2)^\beta|g''(\p(z_n))(\p'(z_n))^2+g'(\p(z_n))\p''(z_n)|\log \frac{2}{1-|\p(z_n)|^2}\\
&-&2C\displaystyle\|\cj\|_{e,\z^2\rightarrow \zb}.
\end{eqnarray*}
Hence
$$\displaystyle\lim_{n\rightarrow\infty}(1-|z_n|^2)^\beta|g''(\p(z_n))(\p'(z_n))^2+g'(\p(z_n))\p''(z_n)|\log \frac{2}{1-|\p(z_n)|^2}\leq C\|\cj\|_{e,\z^2\rightarrow \zb}.$$
On the other hand,  the lower bound  can be easily proved by Lemma 2.4 and 2.5.

If ~$\alpha>2$, the proof is similar to that of (iv) except that we now choose the test function ~$t_n(z)=\displaystyle\frac{(1-|\p(z_n)|^2)^2}{(1-\overline{\p(z_n)}z)^\alpha}$~ instead of ~$O_n(z)$. This completes the proof of Theorem 3.2.

\end{proof}
Using the same methods of Theorem 3.1 and 3.2, we can have the following results.

\begin{theorem}
 Let $\varphi$ be an analytic self-map of $D$  and~$g\in H(D)$, and ~$\ci:\za \rightarrow \zb $~is a bounded operator.

 (i) If ~$0<\alpha<1$~, then
 \begin{equation}
\|\ci\|_{e,\za\rightarrow \zb}\asymp \displaystyle\limsup_{n\rightarrow\infty}(n+1)^\alpha\|((g\circ\p)(\p')^2){\p}^n\|_{\vb}.
\end{equation}

 (ii) If ~$\alpha=1$, then
 \begin{equation}
 \begin{aligned}
\|\ci\|_{e,\z\rightarrow \zb}\asymp \max\{\displaystyle\limsup_{n\rightarrow\infty}(n+1)\|((g\circ\p)(\p')^2){\p}^n\|_{\vb},\\
\displaystyle\limsup_{n\rightarrow\infty}(\log n)\|((g'\circ\p)(\p')^2+(g\circ\p)(\p'')){\p}^n\|_{\vb}\}.
\end{aligned}
\end{equation}

(iii)~~If ~$\alpha>1$~, then
\begin{equation}
\begin{aligned}
\|\ci\|_{e,\za\rightarrow \zb}\asymp \max\{\displaystyle\limsup_{n\rightarrow\infty}(n+1)^\alpha\|((g\circ\p)(\p')^2){\p}^n\|_{\vb},\\
\displaystyle\limsup_{n\rightarrow\infty}( n+1)^{\alpha-1}\|((g'\circ\p)(\p')^2+(g\circ\p)(\p'')){\p}^n\|_{\vb}\}.
\end{aligned}
\end{equation}
  \end{theorem}

\begin{theorem}
Let $\varphi$ be an analytic self-map of $D$  and~$g\in H(D)$~, and ~$\jc:\za \rightarrow \zb $~is a bounded operator.

(i) ~If~$0<\alpha<1$~, then
\begin{equation}
\|\jc\|_{e,\za\rightarrow \zb}=0.
\end{equation}

(ii)~~If $\alpha=1$, then
\begin{equation}
\|\jc\|_{e,\z\rightarrow \zb}\asymp \displaystyle\limsup_{n\rightarrow\infty}(\log n)\|((g'\p'){\p}^n\|_{\vb}\}.
\end{equation}

(iii)~~If~$1<\alpha<2$, then
\begin{equation}
\|\jc\|_{e,\za\rightarrow \zb}\asymp \displaystyle\limsup_{n\rightarrow\infty}( n+1)^{\alpha-1}\|((g'\p'){\p}^n\|_{\vb}\}.
\end{equation}

(iv)~~If~$\alpha=2$, then
\begin{equation}
\|\jc\|_{e,\z^2\rightarrow \zb}\asymp \max\{\displaystyle\limsup_{n\rightarrow\infty}( n+1)\|((g'\p')){\p}^n\|_{\vb},\displaystyle\limsup_{n\rightarrow\infty}(\log  n)\|(g''){\p}^n\|_{\vb}\}.
\end{equation}

(v)~~If~$\alpha>2$, then
\begin{equation}
\|\jc\|_{e,\za\rightarrow \zb}\asymp \max\{\displaystyle\limsup_{n\rightarrow\infty}( n+1)^{\alpha-1}\|((g'\p')){\p}^n\|_{\vb},\displaystyle\limsup_{n\rightarrow\infty}( n+1)^{\alpha-2}\|(g''){\p}^n\|_{\vb}\}.
\end{equation}

  \end{theorem}

\section*{Conflict of Interests}

The authors declare that they have no conflict of interest.

\end{document}